\documentclass{article}

\usepackage{authblk}
\usepackage{amsmath, amssymb, amsthm,bbm}
\usepackage{graphicx,url,enumerate,xcolor,eucal}
\usepackage[lined,boxed,commentsnumbered]{algorithm2e}

\newtheorem{theorem}{Theorem}[section]
\newtheorem{lemma}[theorem]{Lemma}
\newtheorem{proposition}[theorem]{Proposition}

\theoremstyle{definition}
\newtheorem{definition}[theorem]{Definition}
\newtheorem{example}[theorem]{Example}
\newtheorem{remark}[theorem]{Remark}

\newcommand{\I}{\mathbbm{i}\mspace{1mu}}
\newcommand{\zexp}{\mathfrak{exp}\mspace{1mu}}
\newcommand{\zcosh}{\mathfrak{cosh}\mspace{1mu}}
\newcommand{\zsinh}{\mathfrak{sinh}\mspace{1mu}}
\newcommand{\zcos}{\mathfrak{cos}\mspace{1mu}}
\newcommand{\zsin}{\mathfrak{sin}\mspace{1mu}}

\newcommand{\zco}{\mathfrak{C}\mspace{1mu}}
\newcommand{\zdu}{\mathfrak{D}\mspace{1mu}}
\newcommand{\Z}{\mathfrak{Z}}
\newcommand{\G}{\mathfrak{G}}
\newcommand{\M}{\mathbb{M}}
\newcommand{\U}{\mathfrak{U}}
\newcommand{\W}{\mathfrak{W}}

\newcommand{\Mat}{{\rm Mat}}
\newcommand{\CZ}{\mathbb{C}\mathfrak{Z}}

\newcommand{\zdag}{{\hat{\dag}}}

\DeclareMathOperator{\tr}{tr}
\DeclareMathOperator{\Tr}{Tr}

\newcommand{\zm}{\mathbf{0}}
\newcommand{\mI}{\mathbb{I}}

\numberwithin{equation}{section}

\title{Weakly Unitary Zeon Matrices}
\author{G.~Stacey Staples\footnote{Email: sstaple@siue.edu}}
\affil{Department of Mathematics \&
Statistics\\
Southern Illinois University Edwardsville\\
Edwardsville, IL 62026-1653\\
USA}

\date{}  

\begin{document}

\maketitle

\begin{abstract}
In this paper, zeon matrix groups are discussed for the first time.  Extending the concept of unitary complex matrices (i.e., matrices $A$ such that $A^\dagger A=\mI$) to matrices having entries in the complex zeon algebra $\CZ$,  a ``zeon dagger'' operator $\zdag$ is defined and used to declare a zeon matrix $\Xi$ to be {\em unitary} if $\Xi^\zdag\Xi=\mI$. The unitary zeon matrices of order $n$ form a subgroup $\U_n$ of all invertible order-$n$ zeon matrices.   A more interesting notion is that of ``weak unitarity'' of a zeon matrix.  In particular, a zeon matrix $\Xi$ will be said to be {\em weakly unitary} if $\Xi^\zdag\Xi=\mI+\Gamma$, where $\Gamma$ is a diagonal nilpotent zeon matrix.  The weakly unitary matrices of order $n$ form a group $\W_n$, of which $\U_n$ is a subgroup.  It is further shown that the exponential of a graph's nilpotent adjacency matrix has a natural orthogonal decomposition into the sum of a weakly unitary zeon matrix and a nilpotent zeon matrix.  Eigenvalues of the weakly unitary zeon matrix reveal information about the graph's cycles and paths.   
 \\
\\
Keywords: zeons, matrix groups, graph processes, linear operators 
\end{abstract}

\section{Introduction}

Zeon algebras can be thought of as commutative fermion algebras; i.e., their generators commute and square to zero.  They arise as subalgebras of Clifford algebras (of appropriate signature) and have natural applications in combinatorics and graph theory \cite{Feinsilverzeons,NetoSIAM,CAZ}.   
  
In this paper, the concept of unitary complex matrices (i.e., matrices $A$ such that $A^\dagger A=\mI$) is extended to matrices over the complex zeon algebra $\CZ$;  a ``zeon dagger'' operator $\zdag$ is defined and used to declare a zeon matrix $\Xi$ to be unitary if $\Xi^\zdag\Xi=\mI$. The unitary zeon matrices of order $n$ form a subgroup $\U_n$ of all invertible order-$n$ zeon matrices   A much more interesting notion is that of ``weak unitarity'' of a zeon matrix.  In particular, a zeon matrix $\Xi$ is said to be weakly unitary if $\Xi^\zdag\Xi=\mI+\Gamma$, where $\Gamma$ is a diagonal nilpotent zero matrix.  The weakly unitary matrices $\W_n$ of order $n$ form a subgroup of $\U_n$.  

The motivation for considering unitary and weakly unitary zeon matrices is twofold.  First, they arise naturally as exponentials of nilpotent adjacency matrices.  Second, they provide zeon-algebraic analogues of unitary operators, which are of particular importance in quantum probability and quantum computation.

Essential terminology and background on zeon algebras is presented in Section \ref{preliminaries}.  Graph-theoretic results related to zeon algebras are reviewed in Section \ref{NAM}, while Section \ref{processes} illustrates how zeon matrix exponentials reveal paths in sequences of graphs.  

Motivated by the discussion of zeon path matrices, the notions of unitary and weakly unitary zeon matrices are presented in Section \ref{unitary matrices}, where it is shown that when $\Phi$ is the nilpotent adjacency matrix of a graph, the matrix exponential $\zexp(\I \Phi)$, can be decomposed into an orthogonal sum of a weakly unitary matrix and a nilpotent zeon matrix.  Further, the weakly unitary matrix and its eigenvalues reveal information about paths and cycles contained in the graph.  Writing the nilpotent adjacency matrix $\Phi$ of a graph as the Hadamard product of ``fermion adjacency matrices,'' an alternative formulation of $\zexp(\I\Phi)$ is presented, potentially opening a new avenue of potential discovery.  The paper closes with Section \ref{conclusion}, in which some avenues for further investigation are suggested.

\subsection{Preliminaries}\label{preliminaries}
For $n\in\mathbb{N}$,  let $\CZ_n$ denote the complex commutative algebra generated by the collection $\{\zeta_{\{i\}}: 1\le i\le n\}$ along with the scalar $1=\zeta_\varnothing$ subject to the following multiplication rules:
\begin{gather*}
\zeta_{\{i\}}\,\zeta_{\{j\}}=\zeta_{\{i,j\}}=\zeta_{\{j\}}\,\zeta_{\{i\}}\,\,\text{ \rm
for }i\ne j,\text{ \rm and}\\
{\zeta_{\{i\}}}^2=0\,\,\text{ \rm for }1\le i\le n.
\end{gather*}
It is evident that a general element $u\in\CZ_n$ can be expanded as $u=\displaystyle\sum_{I\in2^{[n]}}u_{I}\,\zeta_{I}$, or more simply as $\sum_I u_I\zeta_I$,  where $I\in2^{[n]}$ is a subset of the {\em $n$-set}, $[n]:=\{1,2,\ldots,n\}$, used as a multi-index, $u_{I}\in\mathbb{R}$, and $\displaystyle\zeta_{I}=\prod_{\iota\in I}\zeta_\iota$.  The algebra $\CZ_n$ is called the ($n$-particle) {\em complex zeon algebra}.  

As a vector space, this $2^n$-dimensional algebra has a canonical basis of  {\em basis blades} of the form $\{\zeta_I: I\subseteq [n]\}$.  The null-square property of the generators $\{\zeta_i:1\le i\le n\}$ guarantees that the product of two basis blades satisfies the following:\begin{equation}\label{blade product}
\zeta_I\zeta_J=\begin{cases}
\zeta_{I\cup J}& I\cap J=\varnothing,\\
0&\text{\rm otherwise.}
\end{cases}
\end{equation}
It should be clear that $\CZ_n$ is graded.  For non-negative integer $k$, the {\em grade-$k$ part} of element $u=\sum_I u_I\zeta_I$ is defined as \begin{equation*}
\langle u\rangle_k=\sum_{\{I:|I|=k\}}u_I\zeta_I.
\end{equation*}
An inner product is defined on $\CZ_n$ by linear extension of 
\begin{equation*}
\displaystyle\left<\sum_I u_I\zeta_I,\zeta_J\right>=u_J.
\end{equation*}

The basis blades $\{\zeta_I: I\in 2^{[n]}\}$ form an orthonormal collection with respect to this inner product; thus, any element $u\in\CZ_n$ can be expanded as $\displaystyle u=\sum_{I}\left<u,\zeta_I\right>\zeta_I.$   Moreover, an inner product  norm is naturally induced on $\CZ_n$ by $\|u\|=\langle u,u\rangle^{1/2}$.    

It is often convenient to separate the scalar term of a zeon element from the rest of it. To this end, for $z\in\CZ_n$ we write $\zco z=\langle z\rangle_0$ for the {\em scalar part} of $z$, and $\mathfrak{D}z=z-\zco z$ for the {\em dual part} of $z$.  The dual part of any zeon element is always nilpotent; i.e., $(\zdu z)^k=0$ for some positive integer $k$.  The least such positive integer is called the {\em index of nilpotency} of $\zdu z$ and is denoted by $\kappa(\zdu z)$.

The finite-dimensional algebra is extended to zeon algebras of arbitrary dimension by defining \begin{equation*}
\CZ={\rm span}\{\zeta_I: I\subset\mathbb{N}, I \,\text{\rm finite}\}, 
\end{equation*}
which we call the {\em infinite-dimensional (complex) zeon algebra}.

An element  $u\in CZ$ is uniquely invertible if and only if $\zco u\ne0$, and the inverse  is given by \begin{equation*}
u^{-1}=\frac{1}{\zco u}\sum_{j=0}^{\kappa-1}(-1)^{j}(\zco u)^{-j}(\mathfrak{D}u)^{j}.
\end{equation*}

The maximal ideal consisting of nilpotent zeon elements will be denoted by ${\CZ}^\circ=\{u\in \CZ: \zco u=0\}$. The multiplicative Abelian group of invertible zeon elements is denoted by \begin{equation*}
{\CZ}^\times=\CZ\setminus {\CZ}^\circ=\{u\in \CZ: \zco u\ne 0\}.
\end{equation*}

\subsection{Zeons and Graph Theory}\label{NAM}

For purposes of the discussion at hand, a {\em graph} on $n$ vertices is an ordered pair $G=(V,E)$ of sets.  Elements of $V=\{v_1, \ldots, v_n\}$ are the {\em vertices} of $G$, and elements of $E$ are unordered pairs of distinct vertices, called the {\em edges} of $G$.  Two vertices $v_i,v_j\in V$ are said to be {\em adjacent} if there exists an edge $\{v_i,v_j\}\in E$.  

A {\em walk} $(v_0,\ldots,v_k)$ in a graph $G$ is a finite sequence of vertices in $G$ with {\em initial vertex} $v_0$ and {\em terminal vertex} $v_k$ such that there exists an edge $\{v_j,v_{j+1}\}\in E$ for each $0\le j\le k-1$.  A walk of length $k$, or {\em $k$-walk} contains $k$ edges.  A {\em path} is a walk in which no vertex appears more than once.  A {\em closed walk} is a walk whose initial vertex is also its terminal vertex.  A {\em cycle} is a closed path whose terminal vertex is its initial vertex (i.e., $v_k=v_0$).  A {\em Hamiltonian cycle} is an $n$-cycle in a graph on $n$ vertices; i.e., each vertex in the graph appears exactly once (with the exception of the terminal/initial vertex).

\begin{definition}
Let $G=(V,E)$ be a graph on $n$ labeled vertices $V=\{v_1, \ldots, v_n\}$ and edges $E$ consisting of unordered pairs of vertices.  The {\em nilpotent adjacency matrix of $G$} is defined as the $n\times n$ matrix $\Phi=(\varphi_{ij})$ having entries in ${\Z_n}$ determined as follows:
\begin{equation*}
\varphi_{ij}=\begin{cases}
0&\{v_i, v_j\}\notin E,\\
\zeta_{\{j\}}& \{v_i, v_j\}\in E.
\end{cases}
\end{equation*}
\end{definition}

Modifying Dirac notation, let us define $\langle \zeta_{\{i\}} \vert$ to denote the row vector whose entries are all zero, except the $i$th column, which is $\zeta_{\{i\}}$.  Similarly, $\vert j\rangle$ will denote the column vector with $1$ in the $j$th row and zeros elsewhere.  

Because the generators of $\Z_n$ square to zero, a straightforward inductive argument shows that the nonzero terms of $\left<v_i\vert\Phi^k\vert v_j\right>$ are multivectors corresponding to two types of $k$-walks from $v_i$ to $v_j$:
self-avoiding walks (i.e., walks with no repeated vertices) and walks in which $v_i$ is repeated exactly once at some step but are otherwise self-avoiding.   For $i\ne j$, multiplication by $\zeta_{\{i\}}$ cancels walks of the second type, so that nonzero terms represent paths from $v_i$ to $v_j$.   

When $i=j$, walks of the second type are zeroed in the $k^{\rm th}$ step when the walk is closed. Hence, terms of $\left<v_i\vert \Phi^k\vert v_i\right>$ represent the collection of $k$-cycles based at $v_i$.  We thereby arrive at the following result, as first presented in \cite{zeon laplacian}.  

\begin{description}
\item[Nil-Structure Theorem]
Let $\Phi$ be the nilpotent adjacency matrix of an
$n$-vertex graph $G$, and let $k\in\mathbb{N}$.  For $1\le i\ne j \le n$, 
\begin{equation*}
\left<\zeta_{\{i\}}\vert\Phi^k\vert v_j\right>=\sum_{\genfrac{}{}{0pt}{}{(w_1,\ldots, w_k)\in V^k}{(w_k=v_j)\wedge(m\ne \ell\Rightarrow w_m\ne w_\ell)}}\zeta_{(w_1,\ldots, w_k)}=
\sum_{\genfrac{}{}{0pt}{}{I\subseteq V}{|I|=k+1}}\omega_I \zeta_{I},\end{equation*}
where $\omega_I$ denotes the number of $k$-paths from $v_i$ to $v_j$.  Further,  when $1\le i \le n$,
\begin{equation*}
\left<v_i\vert\Phi^k\vert v_i\right>=\sum_{\genfrac{}{}{0pt}{}{I\subseteq V}{|I|=k}}\omega_I \zeta_{I},\end{equation*}
where $\omega_I$ denotes the number of $k$-cycles on vertex set $I$ based at $v_i\in I$.  
\end{description}

\section{The Nilpotent Adjacency Matrix Exponential}\label{NAMexp}

It is worth noting that any square matrix of order $m$ having entries in $\CZ^\circ$ must be nilpotent by the pigeonhole principle.  For any $X\in\M_m(\CZ^\circ)$, the index of nilpotency of $X$ is denoted by $\kappa(X)$.  The exponential of any nilpotent square matrix $X\in\M_m(\CZ^\circ)$ can therefore be expressed as the finite sum \begin{equation*}
\zexp(X)=\sum_{k=0}^{\kappa(\zdu(X))-1}\frac{X^k}{k!}.
\end{equation*}
Hence, the following zeon elementary functions of $X$ can be computed as finite sums: \begin{align*}
\zcos(X)=\frac{1}{2}(\zexp(\I X)+\zexp(-\I X)),&\hskip20pt\zsin(X)=\frac{1}{2}(\zexp(\I X)-\zexp(-\I X)),\\
\zcosh(X)=\frac{1}{2}(\zexp(X)+\zexp(- X)),&\hskip20pt\zsinh(X)=\frac{1}{2}(\zexp(X)-\zexp(-X)).
\end{align*}

Letting $\Phi\in\M_m({\CZ_m}^\circ)$ be the nilpotent adjacency matrix of a graph $G=(V,E)$ on $m$ vertices, it follows that
\begin{equation*}
\zexp(\Phi)=\sum_{k=0}^{\kappa(\Phi)-1}\frac{1}{k!}\Phi^k=\mI+\Psi,
\end{equation*}
where $\Psi\in\M_m({\CZ_m}^\circ)$.  

The off-diagonal entries of $\Psi$ satisfy
\begin{equation}\label{offdiag}
\langle \zeta_{\{i\}}\vert\Psi\vert j\rangle=\sum_{I\subseteq[n]}\frac{\omega_I}{(|I|-1)!}\zeta_I,\hskip12pt(i \ne j),
\end{equation}
where $\omega_I$ is the number of paths in $G$ from $v_i$ to $v_j$ on vertices indexed by $I$.
The diagonal entries of $\Psi$ satisfy
\begin{equation}\label{diag}
\langle i\vert\Psi\vert i\rangle=\sum_{I\subseteq[n]}\frac{\xi_I}{|I|!}\zeta_I,
\end{equation} where $\xi_I$ is the number of cycles\footnote{Edges are regarded as $2$-cycles.} in $G$ on vertices indexed by $I$, and containing $v_i$. 

For these reasons, the matrix $\Psi=\zdu\zexp(\Phi)$ will be referred to as the {\em zeon path matrix} of $G$.

\subsubsection{Elementary Functions and Graph Properties}

\begin{lemma}
Let $G=(V,E)$ be a graph on $n$ vertices with nilpotent adjacency matrix $\Phi$.  Then \begin{enumerate}
\item $G$ is bipartite if and only if  $\tr\zexp(\Phi)=\tr\zcosh(\Phi)$.\label{bipartite cond.}
\item $G$ is Hamiltonian if and only if $\tr (\Phi^n)\ne \zm$, which is true if and only if  $\langle\tr\zexp(\Phi)\rangle_n\ne 0$.\label{Hamiltonian cond.}
\end{enumerate}
\end{lemma}

\begin{proof}
A well known result of graph theory is that $G$ is bipartite if and only if all cycles are of even length, proving \ref{bipartite cond.}.  Item \ref{Hamiltonian cond.} follows immediately from the Nil-Structure Theorem and the definition of the matrix exponential. 
\end{proof}

\subsubsection{Matrix Exponentials and Graph Processes}\label{processes}

Let $\mathcal{G}=\{G_1, G_2, \ldots, G_m\}$ be a collection of graphs on a common fixed set of $n$ vertices, and let $\Phi_1, \ldots, \Phi_m$ be their respective nilpotent adjacency matrices.  Then, $\exists \Psi\in\Mat(n; {\CZ_n}^\circ)$ such that
\begin{equation*}
\zexp(\Phi_1+\cdots+\Phi_m)=\mI+\Psi.
\end{equation*}
Let us refer to any sequence of (not necessarily distinct) graphs from $\mathcal{G}$ as a {\em graph process} on $\mathcal{G}$.   In this case, $\Psi$ represents all paths and cycles in all graph processes on $\mathcal{G}$:
\begin{equation*}
\Psi=\sum_{k=1}^n\sum_{\mathbf{u}=(u_1, \ldots, u_k)\in[m]^k} \frac{1}{k!}\prod_{j=1}^k \Phi_{u_j}.
\end{equation*}

Moreover, the (ordered) products of nilpotent adjacency matrix exponentials enumerate paths in sequences of graphs.  For example, \begin{equation*}
\zexp(\Phi_1)\zexp(\Phi_2)=(\mI+\Psi_1)(\mI+\Psi_2)=\mI+\Psi_1+\Psi_2+\Psi_1\Psi_2,
\end{equation*}
where $\Psi_1$ and $\Psi_2$ enumerate paths and cycles in $G_1$ and $G_2$, respectively, and $\Psi_1\Psi_2$ enumerates paths $v_i\to v_j\to v_k$ where $v_i\to v_j$ are paths in $G_1$ and paths $v_j\to v_k$ exist in $G_2$.   More formally,
\begin{align*}
\langle \zeta_{\{i\}}\vert \Psi_1\Psi_2\vert j\rangle&=\sum_{\ell=1}^n \langle \zeta_{\{i\}}\vert\Psi_1\vert \ell \rangle\langle \ell\vert \Psi_2\vert j \rangle\\
&=\sum_{\ell=1}^n \left(\sum_I \omega_I\zeta_I\right) \left( \sum_{J}\nu_J\zeta_J\right)\\
&=\sum_{K}\varpi_K\zeta_K,
\end{align*}
where for each $\ell$, $\omega_I$ is the number of paths in $G_1$ from $v_i$ to $v_\ell$ on vertices indexed by $I$ and $\nu_J$ is the number of paths in $G_2$ from $v_\ell$ to $v_j$ on vertices indexed by $J$.  Then $\varpi_K$ is the number of paths having initial vertex $v_i$  and terminal vertex $v_j$ on vertices indexed by $K$ in which some number of steps occur initially in $G_1$ and subsequently in $G_2$.   Extending the concept by induction, we arrive at the following proposition.

\begin{proposition}
Let $\vec{\mathcal{G}}=(G_1, \ldots, G_m)$ be a sequence of graphs on fixed vertex set $V$ and let $(\Phi_1, \ldots, \Phi_m)$ be the sequence of nilpotent adjacency matrices associated with $\vec{\mathcal{G}}$.  Defining the matrix $\Omega$ by 
\begin{equation*}
\Omega=\prod_{j=1}^m \zexp(\Phi_j)=\prod_{j=1}^m (\mI+\Psi_j),
\end{equation*}
it follows that \begin{equation*}
\langle \zeta_{\{i\}}\vert \Omega\vert j\rangle=\sum_{K}\varpi_K\zeta_K,
\end{equation*}
where $\varpi_K$ is the number of paths from $v_i$ to $v_j$ on vertices indexed by $K$ in the graph process $\vec{\mathcal{G}}$.  Moreover, 
\begin{equation*}
\langle j\vert \Omega\vert j\rangle=\sum_{K}\xi_K\zeta_K,
\end{equation*}
where $\xi_K$ is the number of cycles based at vertex $v_j$ in the graph process $\vec{\mathcal{G}}$. 
\end{proposition}

It is worth noting that a finite increasing sequence of graphs (i.e., graphs on increasing numbers of vertices) can be modeled with matrices of order $N$ for any upper bound on the number of vertices in the sequence.  For example, the sequence $\vec{\mathcal{G}}=(G_4, \ldots, G_{10})$ seen in Figure \ref{bipartite sequence}, which is made up of bipartite graphs on $n=8, 10, 12, \ldots, 20$ vertices, can be studied with a sequence of order-20 nilpotent adjacency matrices.  

\begin{figure}
$$\fbox{\includegraphics[width=330pt]{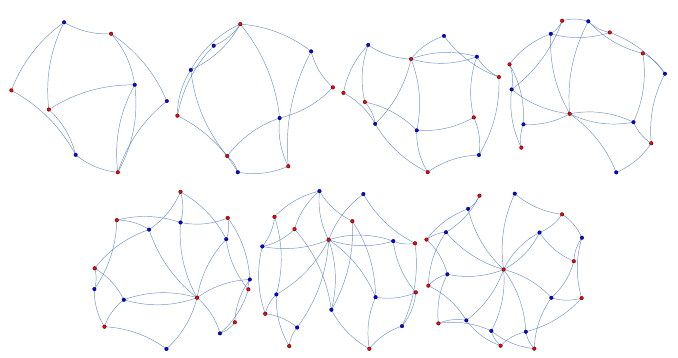}}$$
\caption{A finite increasing sequence of bipartite graphs.\label{bipartite sequence}}
\end{figure}

Obviously, the collection of matrices $\mathfrak{I}_n=\{\mI+\Psi: \Psi\in\M_n({\CZ}^\circ)\}\subset \M_n(\CZ)$ forms a multiplicative subgroup of the order-$n$ invertible zeon matrices, $\G_n$.   The subgroup $\I_n$ motivates our discussion of unitary and weakly unitary zeon matrices.

\section{Unitary and Weakly Unitary Zeon Matrices}\label{unitary matrices}

In order to generalize the concept of the Hermitian conjugate of a complex matrix for use with zeon matrices, it is necessary to find a suitable ``zeon dagger operator.''   To begin, we define the {\em zeon grade involution} of an element $u=\sum_I u_I\zeta_I\in\CZ$ by  $\hat{u}=\sum_I (-1)^{|I|}u_I\zeta_I$.  Note that $\hat{\hat{u}}=u$ and $\widehat{uv}=\hat{u}\hat{v}$ for all $u,v\in\CZ$.  Grade involution extends entrywise to zeon matrices.

As will become evident, the composition of the zeon grade involution with the complex conjugate transpose now provides a suitable operator for defining unitary and weakly unitary zeon matrices.  

\begin{definition}
The {\em zeon dagger operator} $\zdag$ is defined on $\M_{m\times n}(\CZ)$  by $\Xi^\zdag={\widehat{\Xi}}^\dag$, where $\widehat{\Xi}$ is the {\em zeon grade involution} of $\Xi\in\M_{m\times n}(\CZ)$.  
\end{definition}

For any $\Xi\in\M_{m\times n}(\CZ)$,  it is obvious that the operator $\zdag$ satisfies  \begin{equation*}
{\Xi^\zdag}^\zdag= ({\widehat{\Xi}}^\dag)^\zdag=(\widehat{({\widehat{\Xi}}^\dag)})^\dag\\
=(\Xi^\dag)^\dag=\Xi.
\end{equation*}  
Moreover, for matrices $\Psi, \Xi\in\M_n(\CZ)$, we have $(\Psi\Xi)^\zdag=\Xi^\zdag\Psi^\zdag$.

\begin{remark}
The zeon dagger operator restricts to the ordinary conjugate transpose for scalar matrices; i.e.,  $(\zco \Xi)^\zdag=(\zco\Xi)^\dag$ for any $\Xi\in\M_{m\times n}(\CZ)$.
\end{remark} 

Unitary zeon matrices are now defined in the expected way.

\begin{definition}
An invertible zeon matrix $\Xi\in\M_n(\CZ)$ is said to be (strongly) {\em unitary} if \begin{equation*}
\Xi^\zdag\Xi=\Xi\Xi^\zdag=\mI.
\end{equation*}
\end{definition}

We note that the collection $\U_n$ of unitary matrices in $\M_n(\CZ)$ form a multiplicative group, since $\mI\in \U_n$ and for any $\Psi, \Xi\in\U_n$, we have  \begin{equation*}
(\Psi\Xi)^\zdag(\Psi\Xi)=\Xi^\zdag(\Psi^\zdag\Psi)\Xi=\mI.
\end{equation*}
Necessary and sufficient conditions for a zeon matrix to be unitary are established in the next lemma.

\begin{lemma}\label{unitary restrictions}
Let $\Xi\in\M_n(\CZ)$.  Then, $\Xi\in\U_n$ if and only if the matrices $C=\zco\Xi$ and $D=\zdu\Xi$ satisfy the following: \begin{enumerate}
\item $C^\dag C=\mI$, i.e., $C$ is a unitary complex matrix; 
\item $D^\zdag D=\zm$;
\item $C^\dag D=-D^\zdag C$.
\end{enumerate} 
\end{lemma}

\begin{proof}
Letting $C=\zco\Xi$ and $D=\zdu \Xi$, we see that 
\begin{align*}
\Xi^\zdag\Xi=(C+D)^\dag(C+D)&=C^\zdag C+C^\zdag D+D^\zdag C+D^\zdag D\\
&=C^\dag C+C^\dag D+D^\zdag C+D^\zdag D.
\end{align*}
Clearly, $C^\dag C+C^\dag D+D^\zdag C+D^\zdag D=\mI$ if and only if the conditions in the statement of the lemma are satisfied.
\end{proof}

The conditions of Lemma \ref{unitary restrictions} are quite restrictive.  Relaxing these conditions leads to a subgroup of matrices more useful for matters at hand.

\begin{definition}
An invertible zeon matrix $\Xi\in\M_n(\CZ)$ is said to be {\em weakly unitary} if there exists a diagonal nilpotent zeon matrix $\Gamma\in \M_n(\CZ^\circ)$ such that \begin{equation*}
\Xi^\zdag\Xi=\Xi\Xi^\zdag=\mI+\Gamma.
\end{equation*} 
\end{definition}

The weakly unitary zeon matrices of order $n$ clearly form a subgroup $\W_n$ of $\G_n$, the invertible zeon matrices of order $n$.  Further, the unitary matrices $\U_n$ are a subgroup of $\W_n$.  We shall refer to $\W_n$ as the {\em weak unitary group} of $n\times n$ zeon matrices.  

A zeon-valued inner-product is defined on $\M_{m\times n}(\CZ)$ by 
\begin{equation}\label{hdag inner product}
\langle \psi,\omega \rangle_\zdag=\Tr(\omega^\zdag\psi)
\end{equation}
where $\Tr(\Xi)=\frac{1}{n}\tr(\Xi)$ denotes {\em normalized trace} of $\Xi\in\M_n(\CZ)$.  It follows immediately that every weakly unitary zeon matrix $\psi\in\W_n$ satisfies $\langle\psi,\psi\rangle_\zdag=1+\gamma$ for some $\gamma\in\CZ^\circ$.  If $\gamma=0$, then $\psi\in\U_n$.

Let $\Phi$ be the nilpotent adjacency matrix of a graph $G$ on $n$ vertices.  The exponential of $\I\Phi$ reduces to the finite sum
\begin{equation*}
\zexp(\I\Phi)=\displaystyle\sum_{\ell=0}^{\kappa(\Phi)-1}\frac{\I^\ell}{\ell!}\Phi^\ell=\mI+\Psi, 
\end{equation*}
where $\Psi\in\Mat(n; {\CZ_n}^\circ)$.  

Recalling \eqref{offdiag}, the off-diagonal entries of $\Psi$ satisfy
\begin{equation*}
\langle \zeta_{\{i\}}\vert\Psi\vert j\rangle=\sum_{I\subseteq[n]}\frac{\I^{|I|-1}\omega_I}{(|I|-1)!}\zeta_I,\hskip12pt(i \ne j),
\end{equation*}
where $\omega_I$ is the number of paths in $G$ from $v_i$ to $v_j$ on vertices indexed by $I$.
Finally, the diagonal entries of $\Psi$ satisfy
\begin{equation*}
\langle i\vert\Psi\vert i\rangle=\sum_{I\subseteq[n]}\frac{\I^{|I|}\xi_I}{|I|!}\zeta_I,
\end{equation*} where $\xi_I$ is the number of cycles in $G$ on vertices indexed by $I$, and containing $v_i$. 

For convenience, we define the rank-one projector $\varpi_j$ on ${\CZ}^n$ by $\varpi_j=\vert j\rangle\langle j\vert$ for $1\le j\le n$.  We further define $\Delta_n={\rm diag}(\zeta_{\{1\}},\ldots, \zeta_{\{n\}})\in\M_n(\CZ_n)$ to be the diagonal matrix whose nonzero entries are the generators of $\CZ_n$.  With these notational preliminaries in mind, we can state our next result.

\begin{theorem}\label{orthogonal decomposition}
Let $\Phi$ be the nilpotent adjacency matrix of a graph $G=(V,E)$ on $n$ vertices. Then, there exist unique matrices $\Psi_1\in\mathfrak{W}_n(\CZ_n)$ and $\Psi_0\in\M_n({\CZ_n}^\circ)$ such that $\zexp(\I\Phi)=\Psi_1 +\Psi_0$.  In particular, $\Psi_1$ is given by 
\begin{equation}
\Psi_1=\sum_{j=1}^n\varpi_j\zexp(\I\Phi){\varpi_j}^\dag+\Delta_n\zdu(\zexp(\I\Phi)).
\end{equation}
Moreover, the decomposition is orthogonal:  \begin{equation*}
\langle \Psi_0,\Psi_1\rangle_\zdag=\langle \Psi_1,\Psi_0\rangle_\zdag=0.
\end{equation*}
\end{theorem}

\begin{proof}
Writing $\zexp(\I\Phi)=(\phi_{\ell j})$, we see that \begin{equation*}
\sum_{j=1}^n\varpi_j\zexp(\I\Phi){\varpi_j}^\dag={\rm diag}(\phi_{11},\phi_{22},\ldots, \phi_{nn}),
\end{equation*}
the diagonal matrix whose main diagonal is that of $\zexp(\I\Phi)$.  Further, the matrix $\Delta_n\zdu(\zexp(\I\Phi))$ is obtained by multiplying the off-diagonal elements of the path matrix by zeon generators corresponding to initial vertices of paths.  Since the nilpotent main diagonal elements correspond to cycles, such initial vertices are already accounted for and hence zeroed out by $\Delta_n$. Hence, all nonzero entries of $\Delta_n\zdu(\zexp(\I\Phi))$ correspond to paths in the graph $G$.  In accordance with previous results, when $i\ne j$, 
\begin{equation*}
\langle i\vert \Psi_1\vert j\rangle=\sum_{I\subseteq[n]}\frac{\I^{|I|-1}\omega_I}{(|I|-1)!}\zeta_I
\end{equation*}
where $\omega_I$ is the number of paths in $G$ from $v_i$ to $v_j$ on vertices indexed by $I$.
Moreover, \begin{equation*}
\langle j\vert\Psi_1\vert j\rangle=1+\sum_{I\subseteq[n]}\frac{\I^{|I|}\xi_I}{|I|!}\zeta_I,
\end{equation*}
where $\xi_I$ is the number of cycles in $G$ based at vertex $v_j$ on vertices indexed by $I$.
To see that $\Psi_1$ is weakly unitary, we consider the matrix product ${\Psi_1}^\zdag\Psi_1$. Writing $\Psi_1=(y_{ij})$, we see that when $i\ne j$, the $i$th row and $j$th column of  ${\Psi_1}^\zdag\Psi_1$ is
\begin{equation*}
\langle i\vert{\Psi_1}^\zdag\Psi_1\vert j\rangle =\sum_{\ell=1}^n \overline{y_{\ell i}}y_{\ell j}=0,
\end{equation*}
since $\zeta_{\ell}$ annihilates each term of $y_{\ell j}$ and $y_{i \ell}$.  Further, diagonal entries of  ${\Psi_1}^\zdag\Psi_1$ are given by \begin{equation*}
\langle j\vert {\Psi_1}^\zdag\Psi_1\vert j\rangle = \sum_{\ell=1}^n \overline{y_{\ell j}}y_{\ell j}=1+2\sum_{I}\frac{\I^{|I|}\xi_I}{|I|!}\zeta_I,
\end{equation*}
where $\xi_I$ is the number of cycles in $G$ based at vertex $v_j$ on vertices indexed by $I$.  Thus, ${\Psi_1}^\zdag\Psi_1=\mI+\Gamma$ where $\Gamma$ is a diagonal matrix over $\CZ^\circ$.

Let $\Psi_0=\zexp(\I\Phi)-\Psi_1=(z_{ij})$.  We first observe that $z_{jj}=0$ for $j=1, \ldots, n$.  Further, for $i\ne j$, \begin{equation*}
\langle i\vert\Psi_0\vert j\rangle=z_{ij}=(1-\zeta_{\{i\}})\sum_{I\subseteq[n]}\frac{\I^{|I|-1}\omega_I}{(|I|-1)!}\zeta_I,
\end{equation*}
where $\omega_I$ is the number of paths in $G$ from $v_i$ to $v_j$ on vertices indexed by $I$.  In particular, $z_{ij}\zeta_{\{j\}}=0$ for $j=1, \ldots, n$.  Recalling that $\zeta_{\{j\}}y_{\ell j}=0$ for $\ell=1, \ldots, n$.   It follows immediately that 
\begin{align*}
\langle \Psi_1,\Psi_0\rangle_\zdag&=\Tr({\Psi_0}^\zdag\Psi_1)\\
&=\sum_{j=1}^n \sum_{\ell=1}^n\overline{z_{\ell j}}y_{\ell j}=0.
\end{align*}
Similar reasoning shows $\langle \Psi_0,\Psi_1\rangle_\zdag=0$.
\end{proof}

\begin{remark}
Vertices of the graph $G$ are in one-to-one correspondence with states $\varpi_j$ such that $\langle \varpi_j\vert\Psi_1\vert\varpi_j\rangle_\zdag$ is the expectation of $\Psi_1$ in the state $\varpi_j$.  Moreover, this expectation is expressed in terms of the numbers of cycles and their supporting vertices in the graph $G$.
\end{remark}

As seen in \cite{zeon spectral theorem}, $\lambda\in\CZ$ is an eigenvalue of zeon matrix $\Xi\in\M_n(\CZ)$ if and only if $\zco \lambda$ is a {\em simple} zero of the characteristic polynomial of the scalar matrix $\zco(\Xi)$.  If, on the other hand, $\lambda\in\CZ$ is a zero of $\chi_\Xi(u)$ and $\zco\lambda$ is a multiple zero of $\chi_{\zco(\Xi)}(u)$, then $\chi_\Xi(u)$ has infinitely many zeon zeros.  In particular, we have the following result.

\begin{proposition}\label{zeros of chiephi}
Let $\Phi$ be the nilpotent adjacency matrix of a graph $G=(V,E)$ on $n$ vertices.  Writing
$\zexp(\I\Phi)=\Psi_1+\Psi_0$ where $\Psi_1\in\W_n(\CZ_n)$, it follows that for each $j=1, \ldots, n$, the diagonal entry $\langle j\vert\zexp(\I\Phi)\vert j\rangle=\langle j\vert\Psi_1\vert j\rangle$ is a zero of the characteristic polynomial $\chi_{\Psi_1}(u)$.  Moreover, $\chi_{\Psi_1}(u)$ has infinitely many zeon zeros.
\end{proposition}

\begin{proof}
Fixing $j\in\{1, \ldots, n\}$ and setting $\lambda=\langle j\vert \Psi_1\vert j\rangle$, it is sufficient to show that $\vert\lambda\mI-\Psi_1\vert=0$.  Write $\lambda\mI-\Psi_1=(a_{\ell k})$ and consider cofactor expansion in the $j$th column: \begin{equation*}
\vert \lambda\mI-\Psi_1\vert=\begin{vmatrix}
a_{jj}-a_{11}&-a_{12}&\cdots&-a_{1j}&\cdots&-a_{1n}\\
-a_{21}&a_{jj}-a_{22}&\cdots&-a_{2j}&\cdots&-a_{2n}\\
\cdots&\cdots&\cdots&0&\cdots&-a_{jn}\\
-a_{n1}&-a_{n2}&\cdots&-a_{nj}&\cdots&a_{jj}-a_{nn}
\end{vmatrix}.
\end{equation*}
Note that $a_{\ell j}$ annihilates $a_{i \ell}$ for all $\ell\ne j$  since each is a sum over paths including the common vertex $\ell$, and that $a_{jj}=0$.  Hence, all terms of the cofactor expansion of the determinant are zero. 
\end{proof}

In \cite{zeon laplacian eigenvalues}, it was shown that eigenvalues of a graph's  zeon combinatorial Laplacian reveal information about the graph's cycles.  Not surprisingly, eigenvalues of weakly unitary components of nilpotent matrix exponentials can reveal the same information.

\begin{theorem}\label{Upsi eigenvalues}
Let $\Phi$ be the nilpotent adjacency matrix of a graph $G=(V,E)$ on $n$ vertices.  Let $u=(z_1, \ldots, z_n)$ be a column vector in $\mathbb{C}^n$ such that $u^\dag u=1$ and $z_i\ne z_j$ for $i\ne j$.  Let $U$ denote the rank-one projector $uu^\dag$ and let $\zexp(\I\Phi)=\Psi_1+\Psi_0$ be the orthogonal decomposition seen in Theorem \ref{orthogonal decomposition}.  Then, $U\zexp(\I\Phi)=U\Psi_1+U\Psi_0$, where $U\Psi_1$ is a weakly unitary zeon matrix having exactly $n$ distinct zeon eigenvalues  $\lambda_1, \ldots, \lambda_n\in \CZ_n$, such that $\zco\lambda_j=z_j$ for $j=1, \ldots, n$. Moreover, the eigenvalues of $U\Psi_1$ are precisely the diagonal entries $\lambda_j=\langle j\vert U\Psi_1\vert j\rangle$: \begin{equation*}
\lambda_j= z_j\langle j\vert U\Psi_1\vert j\rangle =z_j\sum_{I\subseteq[n]}\frac{\I^{|I|}\xi_I}{|I|!}\zeta_I,
\end{equation*}
where $\xi_I$ is the number of $|I|$-cycles in $G$ based at vertex $v_j$, on vertices indexed by $I$.  
\end{theorem}

\begin{proof}
Observing that $\zco(U\Psi_1)=U$, it follows that $U\Psi_1$ has $n$ distinct zeon eigenvalues $\lambda_1, \ldots, \lambda_n$ whose scalar parts are $\zco \lambda_j=z_j$ for $j=1, \ldots, n$.  

Write $\mathfrak{A}=U\Psi_1=(a_{ij})$, so that $a_{ij}=\langle i\vert U\Psi_1\vert j\rangle$.  Fixing $j\in\{1, \ldots, n\}$, consider the determinant\begin{align*}
\chi_{\mathfrak{A}}(a_{jj})&=\vert a_{jj}\mI-\mathfrak{A}\vert\\
&=\begin{vmatrix}
a_{jj}-a_{11}&-a_{12}&-a_{13}&\cdots&\cdots&-a_{1n}\\
-a_{21}&a_{jj}-a_{22}&-a_{23}&\cdots&\cdots&-a_{2n}\\
\cdots &\cdots&\cdots &\cdots&\cdots&\cdots&\\
-a_{j1}&-a_{j2}&\cdots&0&\cdots&-a_{jn}\\
\cdots &\cdots&\cdots &\cdots&\cdots&\cdots&\\
-a_{n1}&-a_{n2}&\cdots&\cdots&\cdots&a_{jj}-a_{nn}
\end{vmatrix}.
\end{align*}

Computing the determinant by cofactor expansion in the $j$th row, we note first that the cofactor of 0 is determined by the only invertible minor in the determinant.  All other summands in the cofactor expansion are nilpotent.  Secondly, we note that $\zeta_{\{\ell\}}a_{j\ell}=0$ for each $\ell$ because terms of $a_{j\ell}$ represent paths from $v_j$ to $v_\ell$ in $G$.  Computing the cofactor $(-1)^{j+\ell}\vert \widehat{\mathfrak{A}_{j\ell}}\vert$  by cofactor expansion of the minor in row $\ell$, we further see that $\zeta_{\{\ell\}}\vert\widehat{\mathfrak{A}_{j\ell}}\vert=0$.  Hence, $(-1)^{j+\ell}a_{j\ell}\vert\widehat{\mathfrak{A}_{j\ell}}\vert=0$ for each $\ell=1, \ldots, n$.

We conclude that $\chi_{\mathfrak{A}}(a_{jj})=\vert a_{jj}\mI-\mathfrak{A}\vert=0$ for $j=1, \ldots, n$, so that the eigenvalues of $\mathfrak{A}=U\Psi_1$ are precisely $\lambda_j=a_{jj}$.
\end{proof}

\begin{remark}
As was the case in Theorem \ref{orthogonal decomposition}, $U\zexp(\I\Phi)=U\Psi_1+U\Psi_0$ is an orthogonal decomposition; i.e.,  $\langle U\Psi_1,U\Psi_0\rangle_\zdag=0$. 
\end{remark}

In the next theorem, an approach is offered to obtain an eigenvalue of $\Psi_1$ while only concerning ourselves with real zeon matrices, i.e., matrices in $\M_n(\Z_n)$.

\begin{theorem}
Let $\Phi$ be the nilpotent adjacency matrix of a graph $G=(V,E)$ on $n$ vertices.  Let $U={\rm diag}(z_1, \ldots, z_n)\in\{\pm 1\}^n$ be a diagonal matrix satisfying $z_j=1$ for some $j\in[n]$ and $z_\ell=-1$ for $\ell\ne j$, and let $\zexp(\I\Phi)=\Psi_1+\Psi_0$ be the orthogonal decomposition seen in Theorem \ref{orthogonal decomposition}.  Then, $U\zexp(\I\Phi)=U\Psi_1+U\Psi_0$, where $U\Psi_1$ is a weakly unitary zeon matrix having unique eigenvalue  $\lambda_j=\langle j\vert U\Psi_1\vert j\rangle = \langle j\vert \Psi_1\vert j\rangle\in \CZ_n$.  This eigenvalue is also a zero of the characteristic polynomial of $\Psi_1$ and is given by \begin{equation*}
\lambda_j=\sum_{I\subseteq[n]}\frac{\I^{|I|}\xi_I}{|I|!} \zeta_I,
\end{equation*}
where $\xi_I$ is the number of $|I|$-cycles in $G$ based at vertex $v_j$, on vertices indexed by $I$.   Moreover, a unit right eigenvector $\vert \mu\rangle=(\mu_1, \ldots,\underset{j\text{\rm th pos}}{\underbrace{1}},\ldots \mu_{n})^\intercal$ is given by 
\begin{equation*}
\mu_\ell=
-\frac{1}{2}\langle \ell \vert U\Psi_1\vert j\rangle=-\frac{1}{2}\sum_{I\subseteq[n]}\frac{\I^{|I|-1}\omega_I}{(|I|-1)!}\zeta_I,
\end{equation*}
where $\omega_I$ denotes the number of paths from $v_\ell$ to $v_j$ in $G$ on vertices indexed by $I$.
\end{theorem}

\begin{proof}
The fact that $\lambda_j$ is an eigenvalue follows from the proof of Theorem \ref{Upsi eigenvalues}.  Proceeding to the characterization of the associated eigenvector, let $\vert \mu\rangle=(\mu_1, \ldots, 1, \ldots, \mu_{n})^\intercal$ as stated above, and let us consider $U\Psi_1\vert \mu\rangle$.  Note that paths $k\to \ell$ and paths $\ell\to j$ are annihilated by $\zeta_\ell$, so it follows that $\langle k\vert U\Psi_1\vert\ell\rangle\langle \ell\vert U\Psi_1\vert j\rangle=0$ when $\ell\ne j$.  Thus, the $k$th component of $U\Psi_1\mathbf{v}$ is  given by 
\begin{align*}
\langle k\vert U\Psi_1\vert \mu\rangle &=\lambda_j \langle k\vert U\Psi_1\vert j\rangle +\frac{1}{2}\sum_{\ell\ne j}\langle k\vert U\Psi_1\vert\ell\rangle\langle \ell\vert U\Psi_1\vert j\rangle\\
&=\lambda_j \mu_k,
\end{align*}
which shows that $U\Psi_1\vert \mu\rangle=\lambda_j\vert \mu\rangle$.
\end{proof}

\begin{example}
Letting $\Phi$ be the nilpotent adjacency matrix of the graph in Figure \ref{eigenvalue via ampliation}, writing $\zexp(\I\Phi)=\Psi_1+\Psi_0$, and letting $U={\rm diag}(1,-1,-1,-1,-1,-1)$, we have a unique eigenvalue of $U\Psi_1$ satisfying $\zco\lambda_1=1$:
\begin{align*}
\lambda_1=\langle 1\vert U\Psi_1\vert 1\rangle&=1-\frac{1}{2} \zeta _{\{1,2\}}-\frac{1}{2} \zeta _{\{1,3\}}-\frac{\I}{3} \zeta _{\{1,2,3\}}+\frac{1}{12} \zeta _{\{1,2,3,5\}}.
\end{align*}

\begin{figure}
$$\includegraphics[width=200pt]{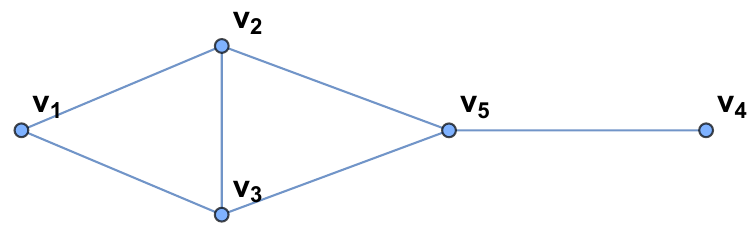}$$
\caption{A five vertex graph.\label{eigenvalue via ampliation}}
\end{figure}
The eigenvector corresponding to $\lambda_1$ is 
\begin{equation*}
\vert \mu\rangle=\begin{pmatrix}
 1 \\
 -\frac{\I}{2}  \zeta _{\{1,2\}}+\frac{1}{4} \zeta _{\{1,2,3\}}+\frac{\I}{12}  \zeta _{\{1,2,3,5\}} \\
 -\frac{\I}{2} \zeta _{\{1,3\}}+\frac{1}{4} \zeta _{\{1,2,3\}}+\frac{\I}{12} \zeta _{\{1,2,3,5\}} \\
 \frac{\I}{12} \zeta _{\{1,2,4,5\}}+\frac{\I}{12} \zeta _{\{1,3,4,5\}}-\frac{1}{24} \zeta _{\{1,2,3,4,5\}} \\
 \frac{1}{4} \zeta _{\{1,2,5\}}+\frac{1}{4} \zeta _{\{1,3,5\}}+\frac{\I}{6} \zeta _{\{1,2,3,5\}} \\
\end{pmatrix}.
\end{equation*}
The fourth component of $\vert \mu\rangle$,  \begin{align*}
\mu_4&=  \frac{\I}{12} \zeta _{\{1,2,4,5\}}+\frac{\I}{12} \zeta _{\{1,3,4,5\}}-\frac{1}{24} \zeta _{\{1,2,3,4,5\}}\\
&=-\frac{1}{2}\left(-\frac{\I}{6} \zeta _{\{1,2,4,5\}}-\frac{\I}{6} \zeta _{\{1,3,4,5\}}+\frac{1}{12} \zeta _{\{1,2,3,4,5\}}\right)\\
&=-\frac{1}{2}\left(-\frac{\I}{3!} \zeta _{\{1,2,4,5\}}-\frac{\I}{3!} \zeta _{\{1,3,4,5\}}+\frac{2}{4!} \zeta _{\{1,2,3,4,5\}}\right),
\end{align*}
indicates the existence of exactly one path from $v_4$ to $v_1$ on vertices indexed by $\{1,2,4,5\}$, one path on vertices indexed by $\{1,3,4,5\}$, and two paths on vertices indexed by $\{1,2,3,4,5\}$.   
\end{example}

\subsubsection*{Hadamard Decomposition}

Observing that each zeon generator can be written as a product of disjoint fermion annihilators (or creators), we have an obvious decomposition of the nilpotent adjacency matrix into an Hadamard product of fermion adjacency matrices.

Letting $\mathfrak{F}$ denote the algebra of fermion annihilators generated by $\{f_i:i\in\mathbb{N}\}$, zeon generators can be defined by $\zeta_{\{j\}}=f_{2j-1}f_{2j}$ for $j\in\mathbb{N}$.  The nilpotent adjacency matrix $\Phi$ of a graph $G=(V,E)$ on $n$ vertices can then be written as the Hadamard product of two {\em fermion adjacency matrices} associated with $G$.   In particular, $\Phi=\Phi_1\odot\Phi_2$, where $\Phi_1$ and $\Phi_2$ are determined by \begin{align*}
\langle i\vert\Phi_1\vert j\rangle=\begin{cases}
f_{2j-1}& \{v_i,v_j\}\in E,\\
0&\text{\rm otherwise,}
\end{cases}\hspace{12pt}&
\langle i\vert\Phi_2\vert j\rangle=\begin{cases}
f_{2j}& \{v_i,v_j\}\in E,\\
0&\text{\rm otherwise.}
\end{cases}
\end{align*}

It is well known\footnote{See, for example, \cite{liu}.} that the Hadamard product is related to the Kronecker product by 
\begin{equation*}
\Phi_1\odot\Phi_2=J(\Phi_1\otimes\Phi_2)J^\dag,
\end{equation*}
where $J$ is an appropriately constructed selection matrix.  When $\Phi_1, \Phi_2$ are square matrices of order $n$, $J$ is the $n\times n^2$ matrix \begin{equation*}
J=\begin{pmatrix}
\mathbf{e}_1&\mathbf{e}_{n+2}&\mathbf{e}_{2n+3}&\cdots&\mathbf{e}_{n^2}
\end{pmatrix}^\intercal,
\end{equation*}
whose $n$ rows $\mathbf{e}_j$ are unit coordinate vectors of $\mathbb{C}^{n^2}$.

The nilpotent (zeon) adjacency matrix is then written in terms of fermion adjacency matrices via $\Phi=J(\Phi_1\otimes \Phi_2)J^\dag$.  It is worth noting that $J^\dag J$ is a rank-$n$ orthogonal projection.  

Letting $\Theta=J^\dag J(\Phi_1\otimes\Phi_2)$, we find that for each positive integer $m\ge2$,  
\begin{align*}
\Phi^m&=(J(\Phi_1\otimes\Phi_2)J^\dag)^m\\
&=J\Theta^m J^\dag.
\end{align*}

Hence, \begin{align*}
\zexp(\I\Phi)=\sum_{\ell=0}^{\kappa(\Phi)}\frac{\I^\ell}{\ell!}\Phi^\ell&=\sum_{\ell=0}^{\kappa(\Phi)}\frac{\I^\ell}{\ell!}J\Theta^\ell J^\dag\\
&=J\left(\sum_{\ell=0}^{\kappa(\Phi)}\frac{\I^\ell}{\ell!}\Theta^\ell \right)J^\dag\\
&=J\zexp(\I\Theta)J^\dag.
\end{align*}

\section{Concluding Remarks}\label{conclusion}

The preliminary investigation of unitary zeon matrices leads to a much larger problem of classifying the various zeon matrix groups and subgroups.  As we have seen, the groups $\U_n$, $\W_n$, and $\G_n$ satisfy $\U_n\le\W_n\le\G_n$.  Further, the group $\mathfrak{I}_n=\{\mI+\Psi: \Psi\in\M_n(\CZ^\circ)\}\le \G_n$.

The relationships among zeon matrix groups, including chains and lattices, are yet to be investigated.  The relationships among weakly unitary matrices obtained from orthogonal decompositions of sequences of nilpotent adjacency matrix exponentials could also be interesting--particularly from a graph-theoretic perspective.

\subsubsection*{Zeon Lie Theory}

The algebra $\M_n(\CZ^\circ)$ of nilpotent zeon matrices constitutes a Lie algebra with the bracket $[X,Y]=XY-YX$.  The exponential map $\zexp:\M_n(\CZ^\circ)\to\G_n(\CZ)$ then satisfies $\zexp(X)= \mI+\Psi$ where $\Psi\in\M_n(\CZ^\circ)$.  In fact, the scalar part of the determinant satisfies $\zco(|\zexp(X)|)=1$ for every $X\in\M_n(\CZ^\circ)$.  

For any nilpotent adjacency matrix $\Phi$, the exponential map satisfies $|\zexp(\I\Phi)|=1$.  Moreover, given any linear combination $\Psi=\sum_{j=1}^m a_j\Phi_j$ of nilpotent adjacency matrices, it can be shown that $|\zexp(\I\Psi)|=1$.  
One obvious question is: What can be said about Lie groups in this context?

\end{document}